\documentclass{amsart}
\usepackage{mathrsfs}
\usepackage{graphicx, subfigure}
\usepackage{latexsym,amsfonts,amssymb}
\usepackage{graphicx,epsfig,subfigure}
\usepackage{psfrag}
\usepackage{enumitem}
\usepackage{colortbl}
\usepackage{verbatim}
\usepackage{hyperref}
\usepackage[all]{xy}
\usepackage{epstopdf}
\usepackage{multirow}

\newtheorem{theorem}{Theorem}[section]
\newtheorem{corollary}[theorem]{Corollary}
\newtheorem{lemma}[theorem]{Lemma}
\newtheorem{conjecture}[theorem]{Conjecture}
\newtheorem{proposition}[theorem]{Proposition}
{\theoremstyle{definition}

\newtheorem{remark}[theorem]{Remark}
\newtheorem{example}[theorem]{Example}
}

\newcommand{\Z}{\mathbb Z}
\newcommand{\R}{\mathbb R}
\newcommand{\C}{\mathbb C}

\newcommand{\codim}{\operatorname{{codim}}}

\title[Linear operators and semialgebraic diffeomorphisms]{Surjectivity of linear operators and semialgebraic global diffeomorphisms}

\author[F. Braun, L.R.G. Dias \MakeLowercase{and} J. Venato-Santos]
{Francisco Braun$^*$, Luis Renato Gon\c{c}alves Dias$^\dagger$ \MakeLowercase{and} Jean Venato-Santos$^\ddagger$}

\address{$^*$ Departamento de Matem\'{a}tica, Universidade Federal de S\~ao Carlos, 
	13565-905 S\~ao Carlos, S\~ao Paulo, Brazil} 
\email{franciscobraun@dm.ufscar.br}

\address{$^{\dagger,\ddagger}$ Faculdade de Matem\'{a}tica, Universidade Federal de Uberl\^{a}ndia, 
	38408-100 Uber\-l\^{a}n\-dia, Minas Gerais, Brazil} 
\email{lrgdias@ufu.br}
\email{jvenatos@ufu.br}



\thanks{$^*$ Partial support provided by the grants 2019/07316-0 and 2020/14498-4 of the S\~ao Paulo Research Foundation (FAPESP)}

\thanks{$^\dagger$ Partial support provided by the Grant 304163/2017-1 of the National Council for Scientific and Technological Development – CNPq}

\thanks{$^\ddagger$ Support provided by the Grant APQ-02056-21 of Fapemig-Brazil}

\begin{document}

\begin{abstract} 
We prove that a $C^{\infty}$ semialgebraic local diffeomorphism of $\R^n$ with non-properness set having codimension greater than or equal to $2$ is a global diffeomorphism if $n-1$ suitable linear partial differential operators are surjective. 
Then we state a new analytic conjecture for a polynomial local diffeomorphism of $\R^n$. 
Our conjecture implies a very known conjecture of Z. Jelonek. 
We further relate the surjectivity of these operators with the fibration concept and state a general global injectivity theorem for semialgebraic mappings which turns out to unify and generalize previous results of the literature. 
\end{abstract}

\maketitle

\section{Introduction}

Let $F=(f_1, \ldots, f_n): \mathbb{K}^n \to\mathbb{K}^n$, $\mathbb{K}=\R$ or $\C$, be a differentiable mapping. 
We denote by $J(F)$ the Jacobian determinant of $F$. 
We say that \emph{$F$ is proper at $y\in \mathbb{K}^n$} if there exists a neighborhood $V$ of $y$ such that $F^{-1}(\overline{V})$ is compact. 
We denote by $S_F$ the set of points at which $F$ is not proper. 
This is the \emph{non-properness set} of $F$, also known as \emph{Jelonek's set} of $F$. 

It is known that $S_F$ is a semialgebraic set when $F$ is a semialgebraic mapping, see for instance \cite[Theorem 6.4]{J}.
Thus, in this case, it makes sense to consider the  dimension and the codimension of $S_F$. 

The following conjecture was raised by Jelonek in \cite{J}: 

\begin{conjecture}[Jelonek's conjecture]\label{c:je}
Let $F: \R^n \to \R^n$ be a polynomial mapping with nowhere vanishing $J(F)$ and such that $\codim S_F\geq 2$. 
Then $F$ is a bijective mapping.  
\end{conjecture}

In his paper, Jelonek proved that Conjecture \ref{c:je} is true under the additional assumption $\codim(S_F)\geq3$, the proof working for the semialgebraic case as well (although the conclusion is just the injectivity of $F$), and also that it is true for $n=2$, see \cite[Proposition 2.13]{BDV_S} for a semialgebraic version in $n=2$.  
Much more appealing is the fact, proved in the same paper, that Conjecture \ref{c:je} implies the famous Jacobian conjecture in $\C^n$:  ``A polynomial mapping $F: \C^n \to \C^n$ with $J(F) =1$ is an automorphism.''

We remark that the Pinchuk \cite{P} mappings $P:\R^2 \to \R^2$ constructed in order to disprove the \emph{real Jacobian conjecture}, that a polynomial mapping $\R^n \to \R^n$ with nowhere zero Jacobian determinant is injective, satisfy $\codim \left(S_P\right) =  1$ (see yet a rational parametrization of $S_P$ for a specific Pinchuk mapping in \cite{Ca}), and so the assumption on $\codim \left(S_F\right) \geq 2$ is necessary in Conjecture \ref{c:je}, and it remains to investigate it in dimension $n\geq3$.

On the other hand, for a fixed $C^{\infty}$ mapping $F=(f_1, \ldots, f_n):\R^n \to \R^n$  (resp. an entire analytic mapping $F=(f_1, \ldots, f_n):\C^n \to \C^n$) and for an $i\in \{1, 2, \ldots, n\}$, let $\Delta_i^F: C^{\infty}(\R^n)  \to C^{\infty}(\R^n)$ (resp. $\Delta_i^F: E_n \to E_n$) be the  operator defined by 
\begin{equation}\label{campos}
\Delta_i^F(g)=J(f_1, \ldots,f_{i-1}, g, f_{i+1}, \ldots, f_n), 
\end{equation}
where $C^{k}(M)$ denotes the space of $C^{k}$ functions defined on a given $C^k$ manifold $M$, and $E_n$ denotes the space of entire analytic functions on $\C^n$ with the topology of uniform convergence on compact subsets. 

Stein \cite{St1, St2, St3}, for $n=2$, and Krasi\'nski and Spodzieja \cite{KS}, for $n\geq 2$, proved that for any polynomial mapping $F:\C^n \to \C^n$ with $J(F)=1$, the image $\Delta_i^F(E_n)$ is dense in $E_n$ for $n-1$ indices $i\in\{1,\ldots, n\}$ if, and only if, $F$ is a polynomial automorphism. 
Actually, for $n=2$, it was proved in \cite{St3} that $\Delta_i^F(E_2)$ is closed, hence $F$ is a polynomial automorphism if and only if $\Delta_i^F$ is surjective for $i=1$ or $2$ in the bidimensional case. 
So the following conjecture, stated for $n=2$ in \cite{St1}, implies the Jacobian conjecture: 
\begin{conjecture}
Let $F:\C^n \to \C^n$ be an entire analytic mapping such that $J(F)=1$. 
Then the image $\Delta_i^F(E_n)$ is dense in $E_n$ for $n-1$ different indices $i\in\{1,\ldots, n\}$. 
\end{conjecture}  

The following analogous statement in the real $C^{\infty}$ case was recently proved false in \cite{BS} for $n\geq 3$ (in $\R^2$ it is true: see for instance \cite{BST}):
\begin{quote}
Let $F: \R^n \to \R^n$ be a $C^{\infty}$ mapping with nowhere vanishing $J(F)$. 
If $\Delta_i^F\left(C^{\infty}(\R^n)\right) = C^{\infty}(\R^n)$ for $n-1$ indices $i\in\{1,\ldots, n\}$, then $F$ is injective. 
\end{quote}

The first main result of this paper is to consider the real $C^{\infty}$ case patching together assumptions on the operators $\Delta_i^F$ and on the set $S_F$ in the following theorem. 
In order to consider $S_F$ we assume $F$ to be semialgebraic. 

\begin{theorem}\label{t:main}
Let $F: \R^n \to \R^n$ be a $C^{\infty}$ semialgebraic mapping with nowhere vanishing $J(F)$ and such that $\codim S_F\geq 2$. 
Then $F$ is a bijective mapping if and only if $\Delta_i^F\left(C^{\infty}(\R^n)\right) = C^{\infty}(\R^n)$ for $n-1$ indices $i\in\{1,\ldots, n\}$. 
\end{theorem}

We remark that in the preceding analogous polynomial case in $\C^n$, the set $S_F$ is empty or a hypersurface, according to \cite[Theorem 3.8]{J1}, hence its \emph{real} codimension is greater than or equal to $2$. 
Thus our assumption in Theorem \ref{t:main} is the natural real semialgebraic counterpart of the preceding complex results. 

We also observe that the counterexample constructed in \cite{BS} is not semialgebraic, but we can (and will) construct a semialgebraic one, see Example \ref{semi}. 
This supports the assumption on $S_F$ in Theorem \ref{t:main}. 

A similar result than the above ones in the polynomial case is Theorem 4.1 of \cite{KS} that for any polynomial mapping $F:\mathbb{K}^n \to \mathbb{K}^n$, $\mathbb{K}=\R$ or $\C$, such that $J(F)=1$, then $F$ is a polynomial automorphism if, and only if, $\Delta_i^F(\mathbb{K}[x_1, \ldots, x_n])=\mathbb{K}[x_1, \ldots, x_n]$ for $n-1$ different indices $i\in\{1,\ldots, n\}$. 
The bijective mapping of $\R^2$ given by $F(x,y)=\big(x(x^2+1), y(y^2+1)\big)$ does not satisfy the hypothesis of this result but does satisfy the ones of Theorem \ref{t:main}. 

Related to the above conjectures, we state the following
\begin{conjecture}
Let $F: \R^n \to \R^n$ be a polynomial mapping with nowhere vanishing $J(F)$ and such that $\codim S_F\geq 2$. 
Then $\Delta_i^F(C^{\infty}(\R^n)) = C^{\infty}(\R^n)$ for $n-1$ indices $i\in\{1,\ldots, n\}$. 
\end{conjecture}
By Theorem \ref{t:main}, this conjecture implies the Jelonek's one \ref{c:je}, and so the Jacobian conjecture as well. 

Our second main result is Theorem \ref{t:gs-geome} below. 
It provides topological properties that certain leaves of foliations related to $F$ must satisfy when $\Delta_{i}^F$ is surjective. 
This theorem is also the main tool for the proof of Theorem \ref{t:main}. 
\begin{theorem}\label{t:gs-geome}
Let $F=(f_1, \ldots, f_n):\R^n \to \R^n$ be a $C^{\infty}$ mapping with nowhere vanishing $J(F)$ and such that there exists $i\in \{1, \ldots, n\}$ with $\Delta_{i}^F(C^{\infty}(\R^n)) = C^{\infty}(\R^n)$. 
Let $j \in \{1, \ldots, n\}\setminus \{ i\}$ fixed. 
Then for  any nonempty connected component $L$ of  $\bigcap_{l\in \{1,\ldots,n\}\setminus \{i, j\}} f^{-1}_l(c_l)$ it holds: 
\begin{enumerate}[label={\textnormal{(\roman*)}}]
\item\label{1t} The non-empty fibers of the restriction $f_{j}|_L: L\to \R$ are connected. 
\item\label{2t} $L$ is diffeomorphic to $\R^2$. 
\end{enumerate}
\end{theorem} 
\begin{remark}
It is clear that the operator $\Delta_i^F$ as defined in \eqref{campos} does not depend on the component $f_i$ of the mapping $F = (f_1, \ldots, f_n)$. 
Actually, by considering $F_i=(f_1, \ldots, f_{i-1}, f_{i+1}, \ldots, f_n)$, we can denote by $\Delta_{i}^{F_i}$ the same operator defined in \eqref{campos}. 
Further, we do not need to assume that the \emph{submersion} $F_i\colon\R^n \to\R^{n-1}$ is part of a local diffeomorphism $F$ in order to state Theorem \ref{t:gs-geome}. 
But the assumption on the surjectivity of $\Delta_{i}^{F_i}$ immediately guarantees the existence of a function $f_i\in C^{\infty}_n$ such that $J(f_1, \ldots,f_{i-1}, f_i, f_{i+1}, \ldots, f_n)\equiv 1$, for instance. 
So we do not miss generality in the statement of Theorem \ref{t:gs-geome}. 
\end{remark}

Actually, for $n=3$, we can characterize the surjectivity of $\Delta_i^F$ as 
\begin{corollary}\label{corollary} 
Let $F = (f_1, f_2, f_3):\R^3 \to\R^3$ be a $C^{\infty}$ mapping with nowhere vanishing $J(F)$ and let $i\in\{1, 2, 3\}$. 
Then $\Delta_{i}^F(C^{\infty}(\R^3)) = C^{\infty}(\R^3)$ if and only if for $j, k \in \{1, 2, 3\}\setminus \{i\}$, $j\neq k$, and for any connected component $L_j$ and $L_k$ of nonempty fibers $f_j$ and $f_k$ respectively, the sets $f_k^{-1}(c) \cap L_j$ and $f^{-1}_j(c) \cap L_k$ are connected for any $c\in \R$. 
\end{corollary}
The sufficient condition above is Theorem 2 of \cite{BS}. 
That this condition is necessary follows directly from \ref{1t} of Theorem \ref{t:gs-geome}. 
See also an $n$-dimensional version of this in \cite{BST}.

The proofs of theorems \ref{t:gs-geome} and \ref{t:main} are presented in Section \ref{s:2}. 

We end the paper in Section \ref{s:3} bringing the concept of local fibration to the subject of the paper. 
In Section \ref{s:3}, we discuss how this topological concept is related to the surjectivity of the operators defined in \eqref{campos}. 
Also, motivated by the works of Byrnes and Lindquist \cite{BL} and Campbell \cite{Camp}, we show that if $F:\R^n \to \R^n$ is  a semialgebraic local homeomorphism such that either $F$ or $F_i$ (as above defined) is a  local fibration on its range, then $F$ is injective, see Theorem \ref{p:image}. 

\section{Proof of theorems \ref{t:main} and \ref{t:gs-geome}}\label{s:2}

For a given mapping $F = (f_1, \ldots, f_n):\R^n \to \R^n$ and for any positive integer $i\leq n$, we will denote by $F_i:\R^n \to\R^{n-1}$ the mapping with coordinate functions $(f_1, \ldots, f_{i-1}, f_{i+1}, \ldots, f_n)$.

We begin gathering the following straightforward properties: 

\begin{lemma}\label{l:level}
Let $F=(f_1, \ldots, f_n):\R^n \to \R^n$ be a local homeomorphism. 
Let $i\in\{1, \ldots, n \}$ fixed and $R$ be a connected component of a nonempty fiber of $F_i:\R^n\to \R^{n-1}$. 
Then
\begin{enumerate}[label={\textnormal{(\alph*)}}]
\item\label{homeo} $R$ is homeomorphic (diffeomorphic if $F$ is $C^1$) to $\R$ and its both ends are unbounded. 
\item\label{monotone} The function $f_{i}$ is strictly monotone along $R$. 
\item\label{conn} If $F$ is $C^{\infty}$, the integral curves of $\Delta^F_{i}$ are the connected components of the nonempty fibers of $F_{i}$.
\end{enumerate}
\end{lemma} 

In particular, for any positive integer $k\leq n$, the non-empty connected components of $\cap_{j=1}^kf_{\sigma(j)}^{-1}(c_j)$ are unbounded for any injective mapping $\sigma: \{1,\ldots, k\} \to \{1,\ldots, n\}$. 
We remark that this could not be the case if $(f_{\sigma(1)}, \ldots, f_{\sigma(k)}): \R^n \to \R^k$ is just a \emph{submersion} (not coming from coordinate functions of a local homeomorphism) with $n\geq3$, see for instance  \cite[Section II]{Costa}: there exists a submersion $g\colon\mathbb{R}^3\to\mathbb{R}^2$ with a fiber having one connected component diffeomorphic to $S^1$. 

A key point in our paper is the following characterization of surjectivity of vector fields given in \cite{DH}.

\begin{lemma}[Part of Theorem 6.4.2 of \cite{DH}]\label{hj}
Let $M$ be a $C^{\infty}$ manifold and $X: C^{\infty}(M)\to C^{\infty}(M)$ be a vector field. 
Then $X(C^{\infty}(M)) = C^{\infty}(M)$, if and only if 
\begin{enumerate}[label={\textnormal{(\roman*)}}]
\item No integral curve of $X$ is contained in a compact subset of $M$. 
\item For all compact $K\subset M$, there exists a compact $K'\subset M$ such that every compact interval on an integral curve of $X$ with endpoints in $K$ is contained in $K'$. 
\end{enumerate}
\end{lemma}

We first provide the 

\begin{proof}[Proof of Statement \ref{1t} of Theorem \textnormal{\ref{t:gs-geome}}] 
	
The proof presented here is an adaptation to a bidimensional submanifold ($L$) of the proof given in \cite[Proposition 1.4]{FGR} in the plane. 
Lemma \ref{l:level} will be used throughout the proof. 
Observe that $L$ is invariant by $\Delta_i^F$: the integral curves of $\Delta_i^F$ in $L$ are given by the connected components of the fibers of $f_j|_L$. 
This is a regular foliation of $L$. 
Below, when referring to the flow of $\Delta_i^F$ we mean this foliation, i.e. the flow of $\Delta_i^F$ restricted to $L$.  

In the proof we will denote by $\gamma_{p} = \gamma_p(t)$ the integral curve of $\Delta_{i}^F$ passing through $p\in L$, i.e., such that $\gamma_p(0) = p$. 
Further, we will say that a continuous path $\sigma$ in $L$ is \emph{tangent} to a given integral curve $\gamma$ at a point $q$ when $\sigma$ is not tranversal to $\gamma$ at $q$. 
Finally, we will say that a continuous injective path $\sigma$ in $L$ is \emph{ideal} if it has at most finitely many points of tangency with the flow of $\Delta_i^F$. 

We begin proving that \emph{any two points of $L$ can be joined by an ideal path}. 
Indeed, for a given $p\in L$, we define $Z(p) = \{q\in L\ |\ \textnormal{ there exists an ideal path from } p \textnormal{ to } q\}$. 
It is enough to prove that $Z(p) = L$. 

By the flow box theorem, \cite[Lemma 3, pag 69]{Camacho}, it follows that $Z(p)$ is non-empty. 
\emph{$Z(p)$ is open:} Indeed, let $x \in Z(p)$ and denote by $\sigma_1$ an ideal path joining $p$ to $x$. 
Applying the flow box theorem to the point $x$, we get a neighborhood of $x$ such that any point $y$ of it can be joined to $x$ by an ideal path $\sigma_2$. 
Clearly the concatenation of $\sigma_1$ and $\sigma_2$ is an ideal path connecting $p$ to $y$. 
\emph{$Z(p)$ is closed:} Indeed, let $\{x_n\}$ be a sequence in $Z(p)$ converging to $x$. 
Similarly as above, for an $n$ big enough, we manage to construct an ideal path from $p$ to $x$, passing by $x_n$, and so $x\in Z(p)$. 
So $Z(p) = L$ as we wanted. 

Now we pass to the proof of the result. 
Acting by contradiction, we suppose that for some $c\in \R$ the set ${f_j|_L}^{-1}(c) = L\cap f_{j}^{-1}(c)$ is not connected. 
We let $R_1$ and $R_2$ be two distinct connected components of ${f_j|_L}^{-1}(c)$ and we consider $p_1 \in R_1$ and $p_2 \in R_2$. 
So $R_k = \gamma_{p_k}$, $k=1,2$. 
We denote by $\Omega(p_1, p_2)$ the (non-empty) set of ideal paths $\sigma: [0,1]\to L$ such that $\sigma(0) = p_1$ and $\sigma(1) = p_2$. 
We then fix a path $\sigma \in \Omega(p_1, p_2)$ that \emph{minimizes the number of tangencies with the flow of $\Delta_i^F$}. 

\emph{We claim that \emph{$R_1 \cap \sigma([0,1]) = \{ p_1\}$ and $R_2 \cap \sigma([0,1]) = \{ p_2\}$}}. 
If, for instance, $R_1 \cap \sigma([0,1]) \neq \{ p_1\}$, let $t_1\in (0,1]$ be the greatest time $t$ such that $\sigma(t) \in R_1$. 
We consider $\lambda$ to be the new path formed by the concatenation of $R_1$ from $p_1$ to $\sigma(t_1)$ and by $\sigma([t_1, 1])$. 
Then by applying the flow box theorem in this interval of $R_1$, we can modify $\lambda$ to a new path $\overline{\sigma} \in \Omega(p_1,p_2)$ without tangency points from $0$ to $\sigma(t)$, for a $t>t_1$ close enough of $t_1$. 
Since $\sigma$ must have a tangency point in $(0,t_1)$, it follows that $\overline{\sigma}$ has less tangencies than $\sigma$, a contradiction with the choice of $\sigma$. 
The proof for $R_2$ is analogous. 
The claim is proved. 
So, in the reasoning below, we can fix our attention to the open interval $(0,1)$. 

Since $f_j\circ \sigma$ has a maximum or a minimum point in $(0,1)$, it follows that the finite set 
$$
\{s \in (0,1)\ |\ \sigma \textnormal{ is tangent to } \gamma_{\sigma(s)} \textnormal{ at some point of } \sigma([0,1])\} 
$$ 
is not empty and we can let $s_0>0$ be its minimum. 
We have two possibilities: 
(i) $\sigma$ is tangent to $\gamma_{\sigma(s_0)}$ at $\sigma(s_0)$ or (ii) $\sigma$ is transversal to $\gamma_{\sigma(s_0)}$ at $\sigma(s_0)$. 

Below we will show that assuming either (i) or (ii) leads to a contradiction. 
This will finish the proof. 

First we assume (i) and let $T:=\{s \in (0,s_0) \mid \# (\gamma_{\sigma(s)} \cap \sigma([0,1])) \geq 2\}$, where as usual $\# Z$ denotes the cardinality of a set $Z$. 
By continuous dependence, it follows that $s\in T$ for all $s<s_0$ close enough to $s_0$. 
Thus $T \neq \emptyset$ and we may consider $\widetilde{s} = \inf T$. 

If $\widetilde s=0$, then $\widetilde s \notin T$. 
If $\widetilde s > 0$, we have that $\gamma_{\sigma(\widetilde s)}$ is transversal to $\sigma([0,1])$ by the choice of $s_0$. 
Then, it follows by the flow box theorem that $\widetilde s \notin T$. 
In any case, so, we have that $\tilde s \notin T$. 

We are going to apply Lemma \ref{hj} to show that $\Delta_i^F$ is not surjective, the aimed contradiction. 
Let the compact set $K:=\sigma([0,1])$ and let $K'$ be any given compact set. 
The integral curve $\gamma_{\sigma(\widetilde s)}$ is not bounded in any direction by Lemma \ref{l:level}. 
So we consider a tubular neighborhood $N$ of it escaping from $K'$ in both directions. 
Then for any $t$ such that $\sigma(t) \in N$ it follows that the integral curve $\gamma_{\sigma(t)}$ stays in $N$ before it can eventually return to $K$. 
In particular, by the definition of $\widetilde s$, there does exist $s \in T$ close enough to $\widetilde s$ such that the orbit $\gamma_{\sigma(s)}$ cuts at least twice the compact $K$ but in between it escapes $K'$. 
By Lemma \ref{hj} it follows that $\Delta_i^F$ is not surjective. 
So (i) cannot be true. 

On the other hand, if (ii) is in force, let $s_1>s_0$ be such that $\sigma(s_1) \in \gamma_{\sigma(s_0)}$ and $\sigma $ is tangent to $\gamma_{\sigma(s_0)}$ at $\sigma(s_1)$. 
We assume that $s_1$ is the smallest $s\in (s_0, 1)$ with this property. 

We define now the continuous path $\widetilde{\sigma}$ being the concatenation of the the following three paths: the path $\sigma([0,s_0])$, the path $\gamma_{\sigma(s_0)}$ from $\sigma(s_0)$ to $\sigma(s_1)$, and the path $\sigma([s_1,1])$. 
By using once more the flow box theorem in a tube around this interval of $\gamma_{\sigma(s_0)}$, we can approximate the path $\widetilde{\sigma}$ to a path $\overline{\sigma} \in \Omega(p_1,p_2)$ with less tangencies with the flow of $\Delta_i^F$ than the path $\sigma$, a contradiction with the choice of $\sigma$. 
\end{proof}

For our next proof we recall the following consequence of \cite[Proposition 2.7]{TZ}:

\begin{proposition}\label{prop:TZ}
Let $M \subset \mathbb{R}^n$ be a $C^{\infty}$ submanifold of dimension $m + 1$ and $g \colon M \to \mathbb{R}^m$ be a $C^{\infty}$ submersion. 
If the fibers of $g$ are diffeomorphic to $\mathbb{R}$ and closed in $\R^n$, then $M$ is diffeomorphic to $\R^{m+1}$.
\end{proposition}
  
\begin{proof}[Proof of Statement \ref{2t} of Theorem \textnormal{\ref{t:gs-geome}}]
By Statement \ref{1t} of the theorem, together with \ref{homeo} of Lemma \ref{l:level}, it follows that the non-empty fibers of the restriction $f_{j}|_L: L\to \R$ are diffeomorphic to $\R$. 
Also, they are clearly closed in $\R^n$. 
Since $f_j|_L(L)$ is an interval (and so diffeomorphic to $\R$), it follows from Proposition \ref{prop:TZ} that $L$ is diffeomorphic to $\R^2$. 
\end{proof}

The following result of \cite{FMS} will be used in the proof of Theorem \ref{t:main}. 
\begin{theorem}[Theorem 1.1 of \cite{FMS}]\label{teo:FMS}
Let $F = (f_1, \ldots, f_n):\R^n \to \R^n$ be a $C^{2}$ semialgebraic mapping with nowhere vanishing $J(F)$ and such that $\codim S_F\geq2$. 
If for all $i,j\in \{1, \ldots, n\}$ with $i\neq j$, any nonempty connected component of  $\bigcap_{l\in \{1,\ldots,n\}\setminus \{i, j\}} f^{-1}_l(c_l)$ is diffeomorphic to $\R^2$, then $F$ is bijective.
\end{theorem}

\begin{proof}[Proof of Theorem \textnormal{\ref{t:main}}]
Assume that $F$ is a global diffeomorphism and let any $i\in \{1,\ldots, n\}$. 
For a given $\psi\in C^{\infty}(\R^n)$, let $\ell(x) = \int_0^{x_i} \psi\circ F^{-1}(z) J (F^{-1})(z) ds$, with $z = (x_1,\ldots, x_{i-1},s,x_{i+1},\ldots, x_n)$. 
Since for any $g \in C^{\infty}(\R^n)$ it holds
$$
\begin{aligned}
\partial_i(g\circ F^{-1})(x) & = J\left((f_1, \ldots, f_{i-1},g,f_{i+1}, \ldots, f_n)\circ F^{-1} \right)(x) \\
& = \Delta_i^F(g) (F^{-1}(x)) J(F^{-1})(x), 
\end{aligned}
$$ 
it follows that defining $g = \ell \circ F \in C^{\infty}(\R^n)$, we have $\Delta^F_i(g) = \psi$. 
This proves the surjectivity of $\Delta_i^F$. 

On the other hand, we now assume that $\Delta_i^F(C^{\infty}(\R^n)) = C^{\infty}(\R^n)$ for $n-1$ indices $i\in\{1,\ldots, n\}$ and let distinct $j ,k \in\{ 1, \ldots, n\}$. 
It follows by \ref{2t} of Theorem \ref{t:gs-geome} that any nonempty connected component of $\bigcap_{l\in \{1,\ldots,n\}\setminus \{i, j\}} f^{-1}_l(c_l)$ is  diffeomorphic to $\R^2$. 
Now the result follows by Theorem \ref{teo:FMS}.  
\end{proof} 

\section{A semialgebraic counterexample}

\begin{example}\label{semi}
The example we are going to construct here is strongly based on the construction presented in \cite{BS}. 
The idea is to consider $C^{\infty}$ semialgebraic functions in place of the $C^{\infty}$ ones appearing in that paper. 
The most difficult part is that we have to modify also the functions $g_h$ of \cite{BS}: this is a family of $C^{\infty}$ functions satisfying a list of properties. 
What happens is that properties (c) and (d) in this list implies that the function $g_h'''$ is flat at the point $1$ without being identically zero. 
This is impossible if $g_h$ is $C^{\infty}$ and semialgebraic, see for instance \cite[Proposition 2.9.5]{BCR}. 
Since the proofs in \cite{BS} use the specific $g_h$, when we modify them, we have to provide new proofs.  

Let $L>0$ and fix $A: \R\to (-L, L)$ and $E: \R \to (0,\infty)$ two $C^{\infty}$ semialgebraic diffeomorphisms. 
For instance, we can take $A(x) = L x/\sqrt{1+x^2}$ and $E(x) = x + \sqrt{1 + x^2}$. 
Further, for any $h > L$, we define the polynomial 
$$
g_h(z) = -\frac{h}{50} z (2z + 3)(2z - 3)(3z^2 + 7). 
$$
Then for $h_2>h_1> L$, we define the $C^{\infty}$ semialgebraic mapping $F: \R^3 \to \R^3$ by 
$$
F(x,y,z) = \left ((A(x) + g_{h_1}(z)) E(y), (A(x) + g_{h_2}(z)) E(y), (1-z^2) E(y) \right). 
$$
We claim that (i) $J(F)$ is nowhere zero, (ii) $F$ is not injective and (iii) $\Delta_1^F$ and $\Delta_2^F$ are surjective. 
Indeed, straightforward calculations give
$$
J(F)(x,y,z) = \frac{h_1-h_2}{50} E(y)^2E'(y) A'(x) \left(36 z^6-59 z^4 + 60 z^2 + 63 \right), 
$$
which is nowhere zero as the polynomial $m(z) = 36 z^6 - 59 z^4 + 60 z^2 + 63$ is positive (to see this, take $z^2 = t$ and observe that the resulting degree $3$ polynomial has derivative strictly positive, and so it has only one real zero, that must be negative as the leading and the constant coefficients of $m(z)$ are positive). 
Moreover, $F(x,y,3/2) = F(x,y,-3/2)$ for any $x,y \in \R$. 
So it remains to prove (iii). 
From Corollary \ref{corollary}, by defining
$$
f_h(x,y,z) = (A(x) + g_{h}(z)) E(y), 
$$ 
with $h>L$, it is enough to prove that for any connected component $L_i$ of $f_i^{-1}(c_i)$, $i = h,3$, it holds that $L_h\cap f_3^{-1}(c_3)$ and $L_3\cap f_h^{-1}(c_h)$ are connected sets. 

This is the most involving part. 
We begin by calculating the connected components of $f_h^{-1}(c)$ and of $f_3^{-1}(c)$ for any $c\in \R$: see the complete frame in Table \ref{ta}, where, similarly as in \cite{BS}, when we write, for instance, $A(x) = - g_h(z)$, $z < -1$, we mean the largest interval contained in $z < -1$ where this expression makes sense. 
{\footnotesize
	\begin{table}[h!]
		\begin{tabular}{|c|c|c|}
			\cline{2-3}
	\multicolumn{1}{c|}{}		& $f_h^{-1}(c)$ & $f_3^{-1}(c)$ \\
			\hline
			\multirow{4}{*}{$c<0$} & $2$ components & $2$ components: \\
			& $A(x) = c E^{-1}(y) - g_h(z)$, $z < 1$ & $E^{-1}(y) = (1-z^2)/c$, $z<-1$ \\ 
			& & \\
			& $A(x) = c E^{-1}(y) - g_h(z)$, $z > 1$ & $E^{-1}(y) = (1-z^2)/c$, $z > 1$ \\
			\hline
			\multirow{6}{*}{$c = 0$} & $3$ components: & $2$ components: \\
			& $A(x) = -g_h(z)$, $z< -1$ & \\
			& & $z = -1$ \\
			& $A(x) = - g_h(z)$, $-1 < z < 1$ & \\
			& & $z = 1$ \\
			& $A(x) = -g_h(z)$, $z > 1$ & \\ 
			\hline 
			\multirow{4}{*}{$c > 0$} & $2$ components: & $1$ component: \\
			& $A(x) = c E^{-1}(y) - g_h(z)$, $z < -1$ & \\
			& & $E^{-1}(y) = (1-z^2)/c$, $-1<z<1$ \\
			&  $A(x) = c E^{-1}(y) - g_h(z)$, $z > -1$ & \\
			\hline 
		\end{tabular}
	\caption{Connected components of fibers}\label{ta}
	\end{table}
}

Clearly the intersection of each of the two connected components of $f_3^{-1}(0)$ with the fibers of $f_h$ are connected or empty (use the fact that the function $y \to c E^{-1}(y) + cte$ is injective). 
Analogously, since $z=-1$ and $z=1$ do not occur both in the same connected component of $f_h^{-1}(c)$, it follows that $f_3^{-1}(0)$ intersected with any of them is also connected. 

Each one of the other required intersections is a solution of the system of equations 
$$
E^{-1}(y) = (1-z^2)/c_3,\ \ \ \ A(x) = c_h E^{-1}(y) - g_h(z), 
$$ 
for $z$ in \emph{only one} of the intervals 
$$
I_1 = (-\infty, -1), \ \ \ \ I_2 = (-1,1),\ \ \ \ I_3 = (1,\infty).
$$ 
They are all connected by 
\begin{lemma}
	Let $\alpha\in \R$ and let $k_{\alpha}: \R \to \R$ be defined by $k_{\alpha} (z) = \alpha (1-z^2) - g_h(z)$. 
	Then the sets $k_{\alpha}^{-1} (-L, L)\cap I_i$, $i = 1,2,3$ are connected. 
	In particular, for each $c_3 \neq 0$, the solution of the system $E^{-1}(y) = (1 - z^2)/c_3$, $A(x) = k_{\alpha}(z)$, $z\in I_i$, form a connected set of $\R^3$ for $i=1,2,3$. 
\end{lemma}
\begin{proof} 
The polynomial function 
$$
k_{\alpha}(z) = \alpha(1-z^2) - g_h(z) = \frac{6 h}{25} z^5 + \frac{h}{50} z^3 - \frac{63 h}{50} z + \alpha (1 - z^2)
$$ 
is such that $k_\alpha'''(z)=\frac{72h}{5}z^2+\frac{3h}{25}>0$, $\forall z$, implying that $k_\alpha''(z)$ is strictly increasing. 
Since $k_{\alpha}''$ is a polynomial of degree $3$, it has exactly one real root, say $q_0$. 
Therefore, $k_\alpha'(z)=\frac{6h}{5}z^4+\frac{3h}{50}z^2-2\alpha z-\frac{63h}{50}$ is strictly decreasing in $(-\infty,q_0)$ and strictly increasing in $(q_0,\infty)$. 
So, since this is a polynomial of degree $4$ with positive leader coefficient and $k_\alpha'(0)=-\frac{63h}{50}<0$, it follows that $k_\alpha'(z)$ has exactly two simple zeros, being one negative, say $p^-$, and the other one  positive, say $p^+$. 
As a consequence, $k_{\alpha}(z)$ has a local maximum at $p^-$ and a local minimum at $p^+$, being strictly increasing in $(-\infty, p^-)\cup (p^+, \infty)$ and strictly decreasing in $(p^-, p^+)$. 
Therefore, since $h > L$ and $k_{\alpha}(\pm 1) = \mp h$, it follows that $k_{\alpha}^{-1}(-L, L)\cap I_i$, $i=1,2,3$, are connected sets. 
\end{proof}
\end{example}

\section{Fibrations and geometric properties of the operators $\Delta_i^F$}\label{s:3}

Let $M$ and $N$ be topological spaces and $f: M\to N$ be a continuous mapping. 
We recall that $f$ is said to be a \emph{local fibration at} $t_0\in N$ if there exist a neighborhood $V$ of $t_0$ and a homeomorphism $h:f^{-1} (t_0)\times V\to f^{-1}\left(V\right)$ such that $f\circ h= \pi_2$, where $\pi_2:f^{-1} (t_0)\times V \to V$ is the canonical projection. 
The set of points of $N$ where $f$ is not a local fibration is called the \emph{bifurcation set of $f$} and it is denoted by $B(f)$. 
We say that $f$ is a \emph{local fibration} if $B(f) = \emptyset$. 
We further say that $f$ is a \emph{local fibration on its range} if $B(f)\cap f(M) = \emptyset$. 
Finally, when it is possible to take $V = N$ for some $t_0\in N$ above, we say that $f$ is a \emph{trivial fibration}. 

Next result provides a relation between the local fibration concept and the surjectivity of the operators $\Delta^F_i$ associated to a given $C^{\infty}$ local diffeomorphism $F:\mathbb{R}^n \to \mathbb{R}^n$: 

\begin{proposition}\label{fibra}
Let $F:\mathbb{R}^n \to \mathbb{R}^n$ be a $C^\infty$ mapping with nowhere vanishing $J(F)$ and $i\in\{1,\ldots,n\}$. 
If $F_i$ is a local fibration on its range then $\Delta_{i}^F \left(C^{\infty}(\R^n)\right) = C^{\infty}(\R^n)$.  
\end{proposition} 
\begin{proof} 
We denote by $\gamma_x$ the integral curve of $\Delta_{i}^F$ passing through $x$. 	
We suppose on the contrary that $\Delta_{i}^F\left(C^{\infty}(\R^n)\right)\neq C^{\infty}(\R^n)$. 
Since from \ref{homeo} and \ref{conn} of Lemma \ref{l:level} the integral curves of $\Delta_{i}^F$  are unbounded, it follows from Lemma \ref{hj} that there exist $p_1, p_2\in \R^n$ and sequences $\{x_k\}\subset \R^n$ and $\{s_k\}, \{t_k\}\subset \R$, with $0< s_k<t_k$, such that 
\begin{equation}\label{56g}
x_k\to p_1,\ \gamma_{x_k}(t_k)\to p_2 \textnormal{ and } |\gamma_{x_k} (s_k)|\to \infty. 
\end{equation}
By using the flow box theorem together with \eqref{56g} (consider also \ref{monotone} of Lemma \ref{l:level}) it follows that $p_1$ and $p_2$ are in distinct integral curves of $\Delta_{i}^F$, say $\gamma_{p_1}$ and $\gamma_{p_2}$, respectively. 
Setting $c = F_{i} (p_1)$ it follows by continuity and by Lemma \ref{l:level} that $\gamma_{p_1}$ and $\gamma_{p_2}$ are distinct connected components of $F_{i}^{-1}(c)$. 
	
Since $F_i$ is a local fibration at $c$, there exist an open neighborhood $V$ of $c$ and a homeomorphism $h: F_{i}^{-1} (c) \times V \to F_{i}^{-1} (V)$ such that $F_i\circ h = \pi_2$. 
The open sets $h\left(\gamma_{p_1}\times V\right)$ and $h\left(\gamma_{p_2}\times V\right)$ do not intersect each other. 
This is a contradiction with \eqref{56g}, because for $k$ big enough the curve $\gamma_{x_k}$ connects these two sets but will also be contained in one of them. 
\end{proof} 

The converse of Proposition \ref{fibra} is not true. 
Let, for instance, the polynomial mapping $F(x,y,z) = \left(P(x,y), z\right)$, where $P: \R^2 \to \R^2$ is the specific above-mentioned Pinchuk mapping considered in \cite{Ca}. 
The operator $\Delta_3^F = J(P)\partial_3$ is clearly surjective, but $F_3$ can not be a fibration at its range because there are points in $P(\R^2)$ with different number of pre-images, according to \cite{Ca}. 
 
Now let $F:\mathbb{K}^n \to \mathbb{K}^n$ be a polynomial mapping for $\mathbb{K}=\C$ (resp. a rational mapping for $\mathbb{K}=\R$) with nowhere vanishing $J(F)$. 
It is proved in \cite{BL} for $\mathbb{K}=\C$ (resp. in  \cite{Camp} for $\mathbb{K}=\R$) that if $F$ is proper when restricted to its image, then $F$ is an invertible mapping. 
I.e., if $S_F\cap F(\mathbb{K}^n) = \emptyset$ then $F$ is invertible. 
Bearing in mind that in this case $B(F) = S_F$ and motivated by these results and by the subject of our paper, we state the following: 
\begin{theorem}\label{p:image} 
Let $F : \R^n \to \R^n$ be a semialgebraic local homeomorphism. 
If either $B(F) \cap F(\R^n) = \emptyset$ or $B(F_i) \cap F_i(\R^n) = \emptyset$ for some $i \in \{1,\ldots n\}$, then $F$ is an injective mapping. 
\end{theorem}

Before proving it, we recall a classical equation involving Euler characteristics, see for example \cite[Corollary 2.5.5]{Di}: 
\begin{lemma}\label{prop:Di}
Let $M$ and $N$ be topological spaces. 
Let $f\colon M\to N$ be a local fibration such that $N$ and $R=f^{-1}(t)$, for some $t\in N$, are homotopy equivalent to finite CW-complexes. 
Then the three Euler characteristics $\chi(M)$, $\chi(N)$ and $\chi(R)$ are defined and related as follows 
$$
\chi(M)=\chi(N)\chi(R). 
$$	
\end{lemma}
A semialgebraic set is homeomorphic to a simplicial complex (which is a CW-complex), see for instance \cite{BCR}. 
So, since the image of a semialgebraic set by a semialgebraic mapping is again a semialgebraic set, we can apply the above formula and deliver the 

\begin{proof}[Proof of Theorem \textnormal{\ref{p:image}}]
For any $y \in F_i(\R^n)$ (resp. $y\in F(\R^n)$), the $d$ connected components of $F^{-1}_i(y)$ (resp. $F^{-1}(y)$) are homeomorphic to $\R$ by \ref{homeo} of Lemma \ref{l:level} (resp. are points). 
By Proposition \ref{prop:Di},
\[
1 = \chi(F_i(\R^n)) d\ \ \ \ (\textnormal{resp. } 1 = \chi(F(\R^n))d ).
\]
Since $\chi(F_i(\R^n)) \in \Z$ (resp. $\chi(F(\R^n)) \in \Z$), it follows that $d=1$. 
This provides the injectivity of $F$ since the $i$-th component of $F$ is monotone along each nonempty fiber of $F_i$ by \ref{monotone} of Lemma \ref{l:level} (resp. trivially). 
\end{proof}

The next result follows direct by Proposition \ref{p:image} since  an injective polynomial $f:\R^n \to \R^n$ is bijective, see \cite{BR}.  

\begin{corollary}  
Let $F=(f_1, \ldots, f_n):\R^n \to \R^n$ be a polynomial mapping with nowhere vanishing $J(F)$. 
If either $B(F) \cap F(\R^n) = \emptyset$ or $B(F_i) \cap F_i(\R^n) = \emptyset$ for some $i \in \{1,\ldots n\}$, then $F$ is invertible. 
\end{corollary}  

In the light of the preceding two results, we finish the paper with an open problem. 
If we do not assume that $F$ is semialgebraic, then Proposition \ref{p:image} may be not true, as the non-injective local diffeormorphism $F(x,y,z) = \left(e^x \cos y, e^x \sin y, z \right)$ illustrates: here $F_3$ is a local fibration on its range. 
It is then natural to ask: 
\begin{quote} 
Is it true that a given $C^{k}$ (resp. $C^0$) mapping $F=(f_1, \ldots, f_n):\R^n \to \R^n$, with nowhere zero $J(F)$ (resp. locally homeomorphism), such that $F_{i}$ is a local fibration on its range for $n-1$ indices $i\in \{1,\ldots, n\}$ is necessarily injective? 
\end{quote}
Since $F_i$ being a locally fibration on its range implies the surjectivity of $\Delta_i^F$ (in the $C^{\infty}$ case), this would also provide an alternative to the false claim recalled in the introduction section ``If $\Delta_i^F(C_n^{\infty}) = C_n^{\infty}$ for $n-1$ indices $i\in\{1,\ldots, n\}$, then $F:\R^n \to \R^n$, a $C^{\infty}$ mapping with nowhere vanishing $J(F)$, is injective.''

\end{document}